\documentclass[10pt,a4paper]{my_m2an}
\usepackage{amsmath,amssymb}
\usepackage[normalem]{ulem}
\usepackage{bm}
\usepackage{tikz}
\newtheorem{theorem}{Theorem}[section]
\newtheorem{lemma}[theorem]{Lemma}

\usetikzlibrary{patterns}
\definecolor{bfonce}{rgb}{0.,0.,0.8}	
\definecolor{bclair}{rgb}{0.87,0.92,1.}
\definecolor{orangec}{rgb}{1.,0.6,0.}

\newcommand{\edge}{\sigma}
\newcommand{\edged}{\epsilon}
\newcommand{\edgeprim}{\tau}
\newcommand{\edges}{{\mathcal E}}
\newcommand{\edgesint}{{\mathcal E}_{{\rm int}}}
\newcommand{\edgesext}{{\mathcal E}_{{\rm ext}}}
\newcommand{\mesh}{{\mathcal M}}
\newcommand{\tdisc}{{\mathcal T}}
\newcommand{\bfn}{{\boldsymbol n}}
\newcommand{\bfu}{{\boldsymbol u}}
\newcommand{\bfx}{{\boldsymbol x}}

\newcommand{\gradi}{{\boldsymbol \nabla}}
\newcommand{\dive}{{\rm div}}
\newcommand{\dt}{\,\mathrm{d}t}
\newcommand{\dx}{\,\mathrm{d}\bfx}
\newcommand{\dedge}{\,\mathrm{d}\gamma}
\newcommand{\exm}{^{(m)}}

\newcommand{\bli}{\begin{list}{-}{\itemsep=1ex \topsep=1ex \leftmargin=0.5cm \labelwidth=0.3cm \labelsep=0.2cm \itemindent=0.cm}}
\newcommand{\Ekin}{E_{\rm k}}
\newcommand{\Gkin}{G_{\rm k}}
\newcommand{\nkedgei}{n_{K,\edge,i}}
\newtheorem{remark}{Remark}[section]
\usepackage[normalem]{ulem}
\newcounter{corr}
\definecolor{violet}{rgb}{0.580,0.,0.827}
\newcommand{\corr}[3]{\typeout{Warning : a correction remains in page
\thepage}
				\stepcounter{corr}        
				{\color{blue}\ifmmode\text{\,\sout{\ensuremath{#1}}\,}\else\sout{#1}\fi}
       {\color{red}#2}
       {\color{violet} #3}}

\begin{document}
\title[Conservative consistent staggered FV methods for the Euler eqns]{Conservativity and weak consistency of a class of staggered finite volume methods \\ for the Euler equations}
\author{R. Herbin}
\address{I2M, CNRS and Universit\'e d'Aix-Marseille, France, \texttt{raphaele.herbin@univ-amu.fr} }

\author{J.-C. Latch\'e}
\address{Institut de Radioprotection et de S\^{u}ret\'{e} Nucl\'{e}aire (IRSN), PSN-RES/SA2I, \\Cadarache, St-Paul-lez-Durance, 13115, France,\texttt{jean-claude.latche@irsn.fr}}
 
\author{S. Minjeaud}
\address{Laboratoire Jean Dieudonn\'e, Universit\'e de Nice-Sophia Antipolis, France, \texttt{sebastian.minjeaud@unice.fr}}

\author{N. Therme}
\address{CEA/CESTA 33116, Le Barp, France, \texttt{nicolas.therme@cea.fr}}

\begin{abstract}
We address here a class of staggered schemes for the compressible  Euler equations~; this scheme was introduced in recent papers  and possesses the following features:  upwinding is performed with respect to the material velocity only and the internal energy balance is solved, with a correction term designed on consistency arguments.
These schemes have been shown in previous works to preserve the convex of admissible states and have been extensively tested numerically.
The aim of the present paper is twofold: we derive a local total energy equation satisfied by the solutions, so that the schemes are in fact conservative, and we prove that they are consistent in the Lax-Wendroff sense.
\end{abstract}


\date{\today}

\maketitle

%
%
\section{Introduction}

The aim of this paper is to prove some conservativity and consistency properties of a class of numerical schemes for the compressible Euler equations which model the flow of an ideal gas and read:
\begin{subequations}
\begin{align}\label{eq:pb_mass} &
\partial_t \rho + \dive( \rho\, \bfu) = 0,
\\[1ex] \label{eq:pb_mom} &
\partial_t (\rho\, \bfu) + \dive(\rho\, \bfu \otimes \bfu) + \gradi p= 0,
\\[1ex] \label{eq:pb_Etot} &
\partial_t (\rho\, E) + \dive(\rho \, E \, \bfu) + \dive ( p \, \bfu )=0,
\\ \label{eq:pb_etat} &
 p=(\gamma-1)\, \rho\, e, \quad E=E_k+e, \quad E_k=\frac 1 2|\bfu|^2,
\end{align} \label{eq:pb}
\end{subequations}
where $t$ is the time, $\rho$, $\bfu$, $p$, $E$,  $e$ and $E_k$ denote the density, velocity, pressure, total energy, internal energy and kinetic energy respectively, and $\gamma > 1$ is the ratio of specific heats.
The notation $|\cdot|$ is used both for the Euclidean norm of a vector of $\xR^d$ or for the absolute value of a real number.
The space--time domain is denoted by $\Omega \times (0,T)$, where $\Omega$ is an open bounded connected subset of $\xR^d$, $1\leq d \leq 3$, and $T< +\infty.$  
System \eqref{eq:pb} is supplemented  by the boundary condition  $\bfu \cdot \bfn=0$ {\it a.e.} on $\partial\Omega$, $\bfn$ denoting the unit normal vector to the boundary, and by initial conditions for $\rho$, $e$ and $\bfu$, denoted by $\rho_0$, $e_0$ and $\bfu_0$ respectively, with $\rho_0 \in L^\infty(\Omega), \rho_0 >0$ and $e_0 \in L^\infty(\Omega), e_0 >0$.
 
\medskip
The use of staggered space discretization for compressible flows is classical: indeed, the well-known Marker-And-Cell (MAC) scheme of \cite{har-65-num,har-71-num}, was followed by numerous works, see e.g. \cite{iss-85-sol,kar-89-pre,mcg-90-sho,bij-98-uni,yoo-99-uni,ver-03-sym,wal-02-sem,kwa-09-met} or \cite{wes-01-pri} for a textbook);  a salient feature of staggered discretisations is their native stability for the numerical simulation of incompressible flows, thus yielding asymptotic preserving schemes in the low Mach number regime.
The decoupling of the fluxes between material velocity and pressure--like part has been implemented in industrial hyprodynamics codes for quite some time for instance in Lagrange-remap schemes, see e.g. \cite{von-50-met,ben-92-com,car-98-con}.
Following this line of thought, some numerical schemes for the Euler equations on general staggered grids have been recently developed \cite{her-14-ons,gra-16-unc,her-18-cons,gas-18-mus} (see also \cite{dak-19-hig} for a similar approach in the Lagrange-projection framework), with share the following features:
\bli
\item a special attention is paid to the discretization of the momentum convection in Equation \eqref{eq:pb_mom} to ensure that discrete solutions satisfy a discrete kinetic energy balance, \ie\ a discrete version of
\begin{equation} \label{eq:pb_ekin}
\partial_t (\rho\, E_k) + \dive(\rho \, E_k \, \bfu) + \bfu.\cdot \gradi p =-\mathcal R, \quad \mathcal R \geq 0,
\end{equation}
where the expression of the discrete remainder $\mathcal R$ is established from the discrete mass and momentum balance equations (see Section \ref{sec:Ekin}.
This quantity $\mathcal R$ may be seen as the dissipation associated to the numerical diffusion.
\item the scheme solves the internal energy balance, \ie\ the relation which may be formally derived at the continuous level by subtracting \eqref{eq:pb_ekin} from \eqref{eq:pb_Etot}:
\begin{equation} \label{eq:pb_eint}
\partial_t (\rho\, e) + \dive(\rho \, e \, \bfu) + p\, \dive \bfu =\mathcal R.
\end{equation}
\end{list}
 These features are also at the heart of a recently developed collocated finite volume scheme \cite{her-19-cel}, and yield an essential outcome: a suitable discretization of \eqref{eq:pb_eint} ensures that $e \geq 0$ and that an entropy inequality  can be derived \cite{gal-19-ent}, with a remarkably unsophisticated definition of the convection fluxes: indeed, upwinding is performed equation per equation, with respect to the material velocity only, so without any exact or approximate solution of Riemann problems at interfaces, thus yielding a natural extension to various equations of state and to more complex problems, such as the reactive Euler equations.
Here, the final goal is to use rather a staggered arrangement of the unknowns to obtain a (class of) schemes that are natively able to cope with all Mach flows \cite{gra-16-unc,her-19-low}; however, passing from the collocated to the staggered framework renders the application of the above-described strategy much more intricate: indeed, the discrete kinetic energy balance is posed on the dual mesh while the internal energy balance is posed on the primal one, and a simple addition of the discrete analogues of \eqref{eq:pb_ekin} and \eqref{eq:pb_eint} is no more possible.
A crucial question when designing the various steps of the scheme is whether the resulting scheme is conservative and consistent;
here in particular, special attention must be paid to the convection operator for the velocity, which is designed in order to respect stability properties only (\ie, as already mentioned, so as to allow to derive a discrete kinetic energy balance).

\medskip
The purpose of this paper is to obtain (positive) answers to these questions: we show that a discrete convection operator on the primal mesh may be derived from the staggered operator for a "cell average" of the face variables (\ie, in this context, either the velocity components or the kinetic energy), and that this convection operator satisfies a Lax-Wendroff consistency property.
To fix ideas, we will work with specific choices for the space and time discretizations, namely an explicit scheme based, with velocity unknowns belonging to the so-called Rannacher-Turek low-order finite element space \cite{ran-92-sim}; however, the arguments invoked below are rather general.
The extension to simplicial meshes, with non conforming P1 velocity unknowns is straightforward.
A similar construction is also possible for the so-called Marker And Cell (MAC) scheme; in particular, the derivation of a convection operator on primal meshes from its staggered analogue is performed, with this latter discretization, in \cite{llo-18-sch}.

\medskip
The schemes studied here are implemented in the open-source software CALIF$^3$S \cite{califs} developed at the French Institut de Radioprotection et de S\^uret\'e Nucl\'eaire (IRSN); they have been extensively tested  numerically in \cite{her-14-ons,gra-16-unc,her-18-cons,gas-18-mus} and are routinely used in nuclear safety applications.
We thus do not reproduce these numerical tests here.
Note that the proof of consistency given below applies in particular to the MUSCL-like scheme which was introduced in \cite{gas-18-mus}.

%
%
\section{Meshes and unknowns}\label{sec:mesh}

Let $\Omega$ be an open bounded polyhedral set of $\xR^d$, $d \ge 1$.
We suppose given a mesh $\mesh$ of $\Omega$, {\it i.e.} a finite collection of compact connected sets $K\in\mesh$, called cells, such that
\[
 \overline{\Omega} = \bigcup_{K\in\mesh} K\qquad\text{and}\qquad \mathring{K}\cap\mathring{L} = \emptyset,\quad \forall K\in\mesh,\ \forall L\in\mesh,\ K\neq L.
\]
The cells are assumed to be intervals (if $d=1$), quadrangles (if $d=2$), or hexahedra (if $d=3$).
In the multidimensional case, the boundary of each cell $K\in\mesh$ is a union of 4 (if $d=2$) or 6 (if $d=3$) parts of hyperplanes of $\xR^d$, which are called faces for short in the following, for $d=2$ as for $d=3$; the boundary $\partial K$ of $K$ reads $\partial K=\cup_{\edge \in \edges_K} \edge$ where $\edges_K$ is the set of the faces of $K$.
We denote by $\edges$ the set of all faces, namely $\edges = \cup_{K \in \mesh} \edges_K $.
We suppose that the mesh is conforming in the usual sense, \ie\ that any cells $K$ and $L$ of the mesh is either disjoint (possibly up to common vertex or edge) or share a whole face.
We denote by $\edgesint$ the set of elements $\edge$ of $\edges$ such that there exist $K$ and $L$ in $\mesh$ ($K \ne L$) such that  $\edge \in \edges_K \cap \edges_L$; such a face $\edge$ is denoted by $\edge=K|L$.
The set of faces located on the boundary of $\Omega$, \ie\ $\edges \setminus \edgesint$, is denoted by $\edgesext$.
For $K \in \mesh$ we denote by $h_K$ the diameter of $K$.
The size of the mesh $\mesh$ is $h_\mesh = \max\{h_K, K \in \mesh\}$.
For $K\in\mesh$ and $\edge \in \edges_K$, we denote by $\bfn_{K,\edge}$ the unit normal vector to $\edge$ outward $K$.
If $A$ is a measurable set of $\xR^d$ or, for $d>1$, of $\xR^{d-1}$, we denote by $|A|$ the Lebesgue measure of $A$.

\medskip
We also introduce now a dual mesh, that is a new partition of $\Omega$ indexed by the elements of $\edges$, namely $\Omega=\cup_{\edge \in \edges} D_\edge$.
For an internal face $\edge=K|L$, the set $D_\edge$ is supposed to be a subset of $K \cup L$ and we define $D_{K,\edge}=D_\edge \cap K$, so that $D_\edge=D_{K,\edge} \cup D_{L,\edge}$ (see Figure \ref{fig:space_disc}); for an external face $\edge$ of a cell $K$, $D_\edge$ is a subset of $K$, and $D_\edge=D_{K,\edge}$.
The cells $(D_\edge)_{\edge \in \edges}$ are referred to as the dual or diamond cells, and $D_{K,\edge}$ as half dual cells or half diamond cells.
For a rectangular or a rectangular paralleliped, we define $D_{K,\edge}$ as the cone having the mass center of $K$ as vertex and the face $\edge$ as basis; this definition is extended to general primal meshes by supposing that $|D_{K,\edge}|$ is still equal to $|K|$ divided by the number of the faces of $K$, denoted by $\zeta$:
\begin{equation} \label{hyp:DKsigma-zeta}
|D_{K,\edge}|=\frac{|K|}\zeta,
\end{equation}
and that the sub-cells connectivities (\ie\ the way the half-dual cells share a common face) is left unchanged.
In one space dimension, $D_{K,\edge}$ is half of the cell $K$ adjacent to $\edge$.
The faces of the dual cells are referred to as dual faces; we denote by $\edges(D_\edge)$ and $\edges(D_{K,\edge})$ the set of faces of $D_\edge$ and $D_{K,\edge}$ respectively (so $\edge \in \edges(D_{K,\edge})$) and by $\edged = \edge|\edgeprim$ the dual face separating $D_\edge$ and $D_{\edgeprim}$.
Note that the actual geometry of the dual cells or dual faces does not need to be specified (and a dual cell may not be a polytope, a dual face being possibly not included in an hyperplane of $\xR^d$), and we will see in the following that it is not required for the definition of the scheme.

\medskip
The scalar unknowns are associated to the cells of the mesh, and read $(\rho_K)_{K\in\mesh}$, $(p_K)_{K\in\mesh}$ and $(e_K)_{K\in\mesh}$ for the density, the pressure and the internal energy respectively.
The velocity unknowns are associated to the faces and read $(\bfu_\edge)_{\edge \in \edges}$, with $\bfu_\edge=(u_{1,\edge},\dots,u_{d,\edge})^t$ the velocity vector associated to $\edge$.

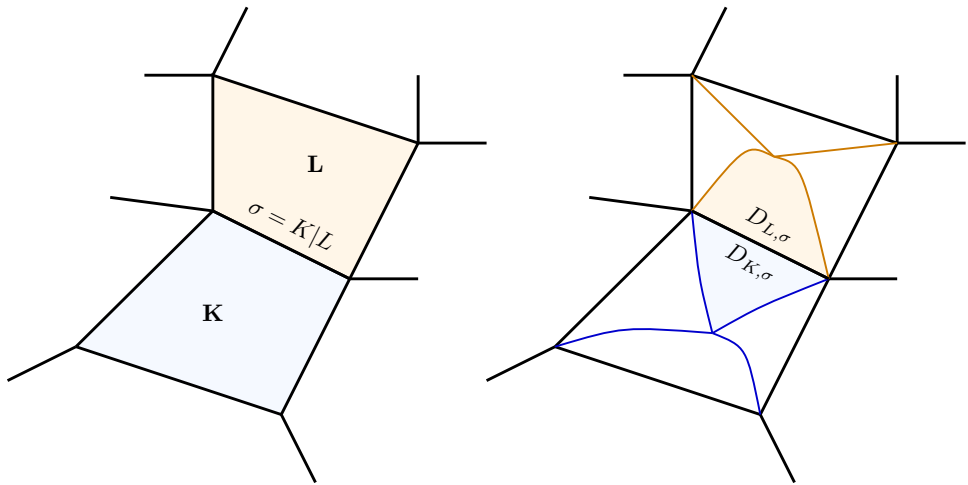
\begin{figure}[h!]
\begin{center}
\scalebox{0.9}{
\begin{tikzpicture} 
\fill[color=bclair,opacity=0.3] (1.,2.) -- (4.,1.) -- (5.,3.) -- (3.,4.) -- (1.,2.);
\fill[color=orangec!30!white,opacity=0.3] (5.,3.) -- (6.,5.) -- (3.,6.) -- (3.,4.);
\draw[very thick, color=black] (1.,2.) -- (4.,1.) -- (5.,3.) -- (3.,4.) -- (1.,2.);
\draw[very thick, color=black] (5.,3.) -- (6.,5.) -- (3.,6.) -- (3.,4.);
\draw[very thick, color=black] (1.,2.) -- (0.,1.5);
\draw[very thick, color=black] (4.,1.) -- (4.5,0.);
\draw[very thick, color=black] (3.,4.) -- (1.5,4.2);
\draw[very thick, color=black] (5.,3.) -- (6,3);
\draw[very thick, color=black] (6.,5.) -- (6.,6.);
\draw[very thick, color=black] (6.,5.) -- (7.,5.);
\draw[very thick, color=black] (3.,6.) -- (3.5,7.);
\draw[very thick, color=black] (3.,6.) -- (2.,6.);
\node at (3., 2.5){$\mathbf K$}; \node at (4.5, 4.7){$\mathbf L$};
\draw[very thick, color=black] (5.,3.) -- (3.,4.) node[midway,sloped,above]{$\edge=K|L$};
\fill[color=bclair,opacity=0.3] (10.,4.) .. controls (10.1,3.) .. (10.3,2.2) .. controls (11.,2.6) .. (12.,3.);
\fill[color=orangec!30!white,opacity=0.3] (10.,4.) .. controls (10.8,5.) .. (11.2,4.8) .. controls (11.6,4.7) .. (12.,3.);
\draw[very thick, color=black] (8.,2.) -- (11.,1.) -- (12.,3.) -- (10.,4.) -- (8.,2.);
\draw[very thick, color=black] (12.,3.) -- (13.,5.) -- (10.,6.) -- (10.,4.);
\draw[very thick, color=black] (8.,2.) -- (7.,1.5);
\draw[very thick, color=black] (11.,1.) -- (11.5,0.);
\draw[very thick, color=black] (10.,4.) -- (8.5,4.2);
\draw[very thick, color=black] (12.,3.) -- (13,3);
\draw[very thick, color=black] (13.,5.) -- (14.,5.);
\draw[very thick, color=black] (13.,5.) -- (13.,6.);
\draw[very thick, color=black] (10.,6.) -- (10.5,7.);
\draw[very thick, color=black] (10.,6.) -- (9.,6.);
\draw[thick, color=bfonce] (10.,4.) .. controls (10.1,3.) .. (10.3,2.2);
\draw[thick, color=bfonce] (10.3,2.2) .. controls (11.,2.6) .. (12.,3.);
\draw[thick, color=bfonce] (8.,2.) .. controls (9.,2.3) .. (10.3,2.2);
\draw[thick, color=bfonce] (11.,1.) .. controls (10.8,2.) .. (10.3,2.2);
\draw[thick, color=orangec!80!black] (10.,4.) .. controls (10.8,5.) .. (11.2,4.8) .. controls (11.6,4.7) .. (12.,3.);
\draw[thick, color=orangec!80!black] (13.,5.) -- (11.2,4.8);
\draw[thick, color=orangec!80!black] (10.,6.) -- (11.2,4.8);
\draw[very thick, color=black] (12.,3.) -- (10.,4.) node[midway,sloped,above]{$D_{L,\edge}$} node[midway,sloped,below]{$D_{K,\edge}$};
\end{tikzpicture}
}
\end{center}
\caption{Mesh and associated notations.}
\label{fig:space_disc}
\end{figure}
%
%
\section{General form of the scheme}\label{sec:scheme}

For the sake of simplicity, we focus here on an explicit-in-time form of the scheme.
Let $(t_n)_{0\leq n \leq N}$, with $0=t_0 < t_1 <\ldots < t_N=T$, define a partition of the time interval $(0,T)$, which we suppose uniform for the sake of simplicity, and let $\delta t=t_{n+1}-t_n$ for $0 \leq n \leq N-1$ be the (constant) time step.
We consider a fractional step scheme, which involves only explicit steps and reads in its fully discrete form, for $0\leq n \leq N-1$:
\begin{subequations}\label{eq:scheme}
\begin{align}
\displaystyle \label{eq:scheme_mass} &
\forall K \in \mesh, \;
\dfrac{|K|}{\delta t}(\rho^{n+1}_K-\rho^n_K) + \sum_{\edge \in\edges(K)} F_{K,\edge}^n=0,
\displaybreak[1]\\ \label{eq:scheme_Eint} &
 \forall K \in \mesh, \;
 \dfrac{|K|}{\delta t}(\rho^{n+1}_K e^{n+1}_K-\rho^n_K e^n_K) + \sum_{\edge \in\edges(K)} F_{K,\edge}^n e^n_\edge
  +|K|\ p^n_K \,(\dive \bfu)^n_K =S^n_K,
 \displaybreak[1]\\[2ex] \label{eq:scheme_eos} &
 \forall K \in \mesh, \;
  p^{n+1}_K=(\gamma-1)\ \rho^{n+1}_K\ e^{n+1}_K,
 \displaybreak[1]\\[2ex]\displaystyle \nonumber &
 \mbox{For } 1 \leq i \leq d,\ \forall \edge \in \edges,
 \\[1ex] \label{eq:scheme_mom} &  \displaystyle \phantom{\forall K \in \mesh, \;}
 \dfrac{|D_\edge|}{\delta t}(\rho^{n+1}_{D_\edge} u^{n+1}_{\edge,i}-\rho^n_{D_\edge} u^n_{\edge,i}) 
 + \sum_{\edged \in\edges(D_\edge)} F_{\edge,\edged}^n u^n_{\edged,i} 
 + |D_\edge|\, (\gradi p)^{n+1}_{\edge,i}
 =0,
\end{align}
\end{subequations}
where the terms introduced for each discrete equation are now defined.
The mass flux $F_{K,\edge}^n$ (equations \eqref{eq:scheme_mass} and \eqref{eq:scheme_Eint}) vanishes on external edges to comply with the boundary conditions and, for $\edge \in \edgesint$, $\edge=K|L$, we have
\begin{equation}
	F_{K,\edge}^n = |\edge|\ \rho_\edge^n \, \bfu_\edge^n \cdot \bfn_{K,\edge},
	\label{primal-mass-flux}
\end{equation}
where $\rho_\edge^n$ is an approximation of the density at the face which is a convex combination of $\rho_K^n$ and $\rho_L^n$, which does not need to be specified further for the matter at hand in this paper.
Similarly, the internal energy $e^n_\edge$ at the face $\edge=K|L$ in Equation \eqref{eq:scheme_Eint} is supposed to be a convex combination of $e^n_K$ and $e^n_L$.
Note that this assumption allows us to deal both with the upwind scheme or with the quasi 2nd-order scheme using a MUSCL-like choice introduced in \cite{gas-18-mus}.
 
In the same relation, the quantity $(\dive \bfu)^n_K$ reads:
\[
(\dive \bfu)^n_K = \frac 1 {|K|} \sum_{\edge \in \edges(K)} |\edge|\ \bfu_\edge^n \cdot \bfn_{K,\edge}.
\]
The pressure gradient in the momentum balance equation \eqref{eq:scheme_mom} is obtained by transposition of the divergence operator with respect to the $L^2$ inner product, and reads for $\edge=K|L$
\[
(\gradi p)^{n+1}_\edge = \frac 1 {|D_\edge|} \ (p_L^{n+1}-p_K^{n+1})\ \bfn_{K,\edge},
\]
and $(\gradi p)^{n+1}_{\edge,i}$ stands for the $i^{\rm th}$ component of $(\gradi p)^{n+1}_\edge$.
The momentum convection operator (\ie\ the first two terms of Equation \eqref{eq:scheme_mom}) are detailed in Section \ref{sec:conv}.
The term $S^n_K$ at the right-hand side of Equation \eqref{eq:scheme_Eint} is a corrective term introduced for consistency, with the aim to compensate the discrete dissipation due to the numerical diffusion in the discrete momentum balance equation \eqref{eq:scheme_mom}; its precise expression is given in Section \ref{sec:Ekin}.
A numerical diffusion term, not featured here, may be introduced in applications in the momentum balance equation to stabilize the scheme.
The initial approximations for $\rho$, $e$ and $\bfu$ are given by the average of the initial conditions $\rho_0$ and $e_0$ on the primal cells and of $\bfu_0$ on the dual cells:
\begin{equation}\label{eq:inicond}
\begin{array}{ll} \displaystyle
\forall K \in \mesh, \quad \rho^0_K = \frac 1 {|K|} \int_K \rho_0(\bfx) \dx,
\quad \mbox{ and } e^0_K = \frac 1 {|K|} \int_K e_0(\bfx) \dx,
\\[4ex] \displaystyle
\forall \edge \in \edges, \quad
\bfu^0_{\edge} = \frac 1 {|D_\edge|} \int_{D_\edge} \bfu_0(\bfx) \dx.
\end{array}
\end{equation}
%
%
\section{The discrete convection on the dual mesh}\label{sec:conv}

Let $z$ be a variable associated to the faces of the mesh, \ie\ to the degrees of freedom $(z_\edge^n)_{\edge\in \edges,\ 0\leq n \leq N}$.
The discrete convection operator for $z$ (\ie\ the discretization of $\partial_t(\rho z) + \dive(\rho z \bfu)$) is defined by:
\begin{equation}\label{eq:conv_d_z}
{\mathcal C}_\edge^n z = \frac{|D_\edge|}{\delta t}\ (\rho_{D_\edge}^{n+1} z_\edge^{n+1}-\rho_{D_\edge}^n z_\edge^n)
+ \sum_{\edged \in \edges(D_\edge)} F_{\edge,\edged}^n\ z_\edged^n,
\end{equation}
where, for $\edged=\edge|\edgeprim$, $z_\edged^n$ is supposed to be a convex combination of $z_\edge^n$ and $z_{\edgeprim}^n$.
The density $\rho_{D_\edge}^n$ at the face $\edge$ is a weighted combination of the density in the neighbouring primal cells, given by, for $0\leq n \leq N$:
\begin{equation}\label{eq:rhoface}
\begin{array}{ll} \displaystyle
\rho_{D_\edge}^n=\frac 1 {|D_\edge|}\ \bigl( |D_{K,\edge}|\ \rho_K^n + |D_{K,\edge}|\ \rho_K^n \bigr), & \mbox{ for } \edge =K|L \in \edgesint 
\\[2ex]
\rho_{D_\edge}^n=\rho_K^n, & \mbox{ for } \edge \in \edgesext, \edge \in \edges(K).
\end{array}
\end{equation}
We suppose that the mass fluxes through the dual edges $\edged$ are obtained thanks to a mass balance over the half-diamond cells, \ie\ satisfy, for $\edge \in \edges(K)$:
\begin{equation}\label{eq:mass_hdiam}
\sum_{\edged \in \edges(D_{K,\edge})} F_{\edge,\edged}^n = \frac 1 \zeta \sum_{\edgeprim \in \edges(K)} F_{K,\edgeprim}^n,
\end{equation}
where we recall that $\zeta$ stands for the number of faces of $K$, so $\zeta =2$ in one space dimension, $\zeta=4$ in two space dimensions and $\zeta=6$ in three space dimensions.
In addition, they are also supposed to satisfy, for $K \in \mesh$, $\edge \in \edges(K)$ and $\edged \subset K$, $\edged \in \edges(D_\edge)$:
\begin{equation}\label{eq:fluxd}
F_{\edge,\edged} = \sum_{\edgeprim \in \edges(K)} \xi_{\edged,\edgeprim}\, F_{K,\edgeprim},
\end{equation}
where the coefficients $\xi_{\edged,\edgeprim}$ are fixed real numbers, \ie\ independent of the cell $K$ and the mesh $\mesh$.
Finally, for diamond cells adjacent to a boundary (which are in fact reduced to a half-diamond cell), one dual face (at least) is an external face, and coincides with a primal face, let us say $\edge$; the mass flux through this face is supposed to vanish (\ie\ to take the same value as the primal mass flux $F_{K,\edge}$, with $K$ the cell adjacent to $\edge$).
The expression of the coefficients $(\xi_{\edged,\edgeprim})$ may be found in \cite{ans-11-anl}  for quadrangles in two space dimensions and in \cite{lat-19-dis} for various meshes (including hexahedra in three space dimensions).

\medskip
The essential outcome of the construction of the dual mass fluxes lies in the fact that \eqref{eq:rhoface} and \eqref{eq:mass_hdiam}, together with the assumption \eqref{hyp:DKsigma-zeta} imply that a mass balance holds over the dual cells:
\begin{equation}\label{eq:mass_diam}
\forall \edge \in \edges,\quad
\frac{|D_\edge|}{\delta t}\ (\rho_{D_\edge}^{n+1} - \rho_{D_\edge}^n) + \sum_{\edged \in \edges(D_\edge)} F_{\edge,\edged}^n = 0.
\end{equation}
%
%
\section{Discrete kinetic energy balance}\label{sec:Ekin}
Thanks to this relation, the scheme satisfies a discrete kinetic energy balance, which is stated in the following lemma (see \cite{her-18-cons} for a proof).

\begin{lemma}[Discrete kinetic energy balance]\label{lmm:kin}
A solution to the system \eqref{eq:scheme} satisfies the following equality, for $1 \leq i \leq d$, $\edge \in \edges$ and $0 \leq n \leq N-1$:
\begin{multline}\label{eq:Ekin}
\dfrac 1 2\, \dfrac{|D_\edge|}{\delta t} \Bigl[ \rho^{n+1}_{D_\edge} (u^{n+1}_{\edge,i})^2-\rho^n_{D_\edge} (u^n_{\edge,i})^2 \Bigr]
+ \dfrac 1 2\ \sum_{\edged \in \edges(D_\edge)} F_{\edge,\edged}^n\ (u^n_{\edged,i})^2
\\
+ |D_\edge|\, (\gradi p)^{n+1}_{\edge,i}\ u^{n+1}_{\edge,i} =
-R^{n+1}_{\edge,i},
\end{multline}
with:
\begin{multline}\label{sw:eq:def_R}
R^{n+1}_{\edge,i}=
\frac 1 2 \frac{|D_\edge|}{\delta t} \rho^{n+1}_{D_\edge} (u^{n+1}_{\edge,i} - u^n_{\edge,i})^2
- \frac 1 2 \sum_{\edged \in \edges(D_\edge)} F^n_{\edge, \edged}\ (u^n_{\edged,i} - u^n_{\edge,i})^2
\\
+ (u^{n+1}_{\edge,i} - u^n_{\edge,i}) \sum_{\edged \in \edges(D_\edge)} F^n_{\edge, \edged}\ (u^n_{\edged,i} - u^n_{\edge,i}).
\end{multline}
\end{lemma} 

The corrective term in the internal energy balance compensates this remainder term:
\begin{equation}\label{eq:def_SK}
\forall K \in \mesh, \qquad S_K^{n+1} = \frac 1 2 \sum_{i=1}^d \sum_{\edge \in \edges(K)} R^{n+1}_{\edge,i}.
\end{equation}
At the first time step, this corrective term is simply set to zero:
\[
 S_K^0 = 0,\quad\forall K\in\mesh.
\]

With the upwind choice for the velocity at the dual face, the remainder term $R^{n+1}_{\edge,i}$ may be proven to be non-negative under a CFL condition.
In this case, $S_K^{n+1}$ is also non-negative, and with a suitable choice of the internal energy at the face (upwind or MUSCL choice with respect to the material velocity), the internal energy remains non-negative by construction of the scheme \cite{gas-18-mus}.
%
%
\section{Returning to the primal mesh}\label{sec:prim_conv}

For $K\in\mesh$, let us define a reconstructed convection operator ${\mathcal C}_K^n$ from the face convection operators ${\mathcal C}_\edge^n$ defined by \eqref{eq:conv_d_z} as follows. 
For a given scalar field $(z_\edge)_{\edge \in \edgesint}$, we set:
\begin{align}\label{eq:conv_p_z_ini}
{\mathcal C}_K^n z 
&= \frac 1 2 \sum_{\edge \in \edges(K)} {\mathcal C}_\edge^n z \nonumber
\\
&= \frac 1 {\delta t} \sum_{\edge \in \edges(K)} \frac {|D_\edge|} 2 (\rho_{D_\edge}^{n+1} z_\edge^{n+1}-\rho_{D_\edge}^n z_\edge^n)
+ \frac 1 2 \sum_{\edge \in \edges(K)} \sum_{\edged \in \edges(D_\edge)} F_{\edge,\edged}^n\ z_\edged^n.
\end{align}
We now define the following quantities (see Figure \ref{fig:conv}):
\begin{align}  \label{eq:defrhoz} &
|K|\ (\rho z)_K^n = \sum_{\edge \in \edges(K)} \frac {|D_\edge|} 2 \ \rho_{D_\edge}^n z_\edge^n,
\\[2ex] \label{eq_defG} &
G_{K,\edge}^n = -\frac 1 2 \sum_{\edged \in \edges(D_\edge), \edged \subset K} F_{\edge,\edged}^n\ z_\edged^n
+ \frac 1 2\sum_{\edged \in \edges(D_\edge), \edged \not \subset K} F_{\edge,\edged}^n\ z_\edged^n.
\end{align}
We easily check that the fluxes $G_{K,\edge}^n$ are conservative, in the sense that, for $\edge=K|L$, $G_{K,\edge}^n=-G_{L,\edge}^n$.
A simple reordering of the summations in \eqref{eq:conv_p_z_ini} using the conservativity of the mass fluxes through the dual edges yields:
\begin{equation}\label{eq:conv_p_z}
{\mathcal C}_K^n z = \frac{|K|}{\delta t}\ \bigl( (\rho z)_K^{n+1} - (\rho z)_K^n\bigr)
+ \sum_{\edge \in \edges(K)} G_{K,\edge}^n.
\end{equation}

\begin{figure}[t!]
\begin{center}
\scalebox{0.9}{
\begin{tikzpicture}[scale=1.5]
\fill[color=bclair!30!white,opacity=0.3] (3.,4.) .. controls (3.1,3.) .. (3.3,2.2) .. controls (4.,2.6) .. (5.,3.);
\fill[color=bclair!30!white,opacity=0.3] (3.,4.) .. controls (3.8,5.) .. (4.2,4.8) .. controls (4.6,4.7) .. (5.,3.);
\draw[very thick, color=black] (1.,2.) -- (4.,1.) -- (5.,3.) -- (3.,4.) -- (1.,2.);
\draw[very thick, color=black] (1.,2.) -- (4.,1.) node[midway,sloped,above]{$K$};
\draw[very thick, color=black] (5.,3.) -- (6.,5.) -- (3.,6.) -- (3.,4.);
\draw[very thick, color=black] (1.,2.) -- (0.,1.5);
\draw[very thick, color=black] (4.,1.) -- (4.5,0.);
\draw[very thick, color=black] (3.,4.) -- (1.5,4.2);
\draw[very thick, color=black] (5.,3.) -- (6.,3);
\draw[very thick, color=black] (7.,5.) -- (6.,5.);
\draw[very thick, color=black] (6.,5.) -- (6.,6.);
\draw[very thick, color=black] (3.,6.) -- (3.5,7.);
\draw[very thick, color=black] (3.,6.) -- (2.,6.);
\draw[thick, color=bfonce] (3.,4.) .. controls (3.1,3.) .. (3.3,2.2) node[near end,sloped,above]{$\edged_1$};
\draw[very thick, color=red!70!black, <-] (3.4,3.1) -- (2.6,2.8) node[midway,sloped,above]{$\ -F_{\edge,\edged_1}$};
\draw[thick, color=bfonce] (3.3,2.2) .. controls (4.,2.6) .. (5.,3.) node[near start,sloped,above]{$\edged_2$};
\draw[very thick, color=red!70!black, <-] (3.9,3.) -- (4.3,2.2) node[midway,sloped,above]{$-F_{\edge,\edged_2}$};
\draw[thick, color=bfonce] (4.2,4.8) .. controls (4.6,4.7) .. (5.,3.) node[midway,sloped,above]{$\edged_3$};
\draw[very thick, color=red!70!black, <-] (5.5,4.4) -- (4.5,3.9) node[midway,sloped,below]{$F_{\edge,\edged_3}$};
\draw[thick, color=bfonce] (3.,4.) .. controls (3.8,5.) .. (4.2,4.8)  node[midway,sloped,above]{$\edged_4$};
\draw[very thick, color=red!70!black, <-] (3.05,4.9) -- (3.75,4.3) node[midway,sloped,below]{$\,F_{\edge,\edged_4}$};
\draw[thick, color=bfonce] (1.,2.) .. controls (2.,2.3) .. (3.3,2.2);
\draw[thick, color=bfonce] (4.,1.) .. controls (3.8,2.) .. (3.3,2.2);
\draw[thick, color=bfonce] (6.,5.) -- (4.2,4.8);
\draw[thick, color=bfonce] (3.,6.) -- (4.2,4.8);
\draw[very thick, color=black] (5.,3.) -- (3.,4.);
\draw[very thick, color=black] (5.,3.) -- (3.,4.) node[midway,sloped,above]{$\edge$};
\draw (7,2.3) node{$G_{K,\edge}=\frac 1 2 (-F_{\edge,\edged_1}z_{\edged_1}-F_{\edge,\edged_2}z_{\edged_2}$};
\draw (7.5,2.) node{$+F_{\edge,\edged_3}z_{\edged_3}+F_{\edge,\edged_4}z_{\edged_4})$};
\end{tikzpicture}
}
\end{center}
\caption{Definition of the convection flux at a primal face.}
\label{fig:conv}
\end{figure}
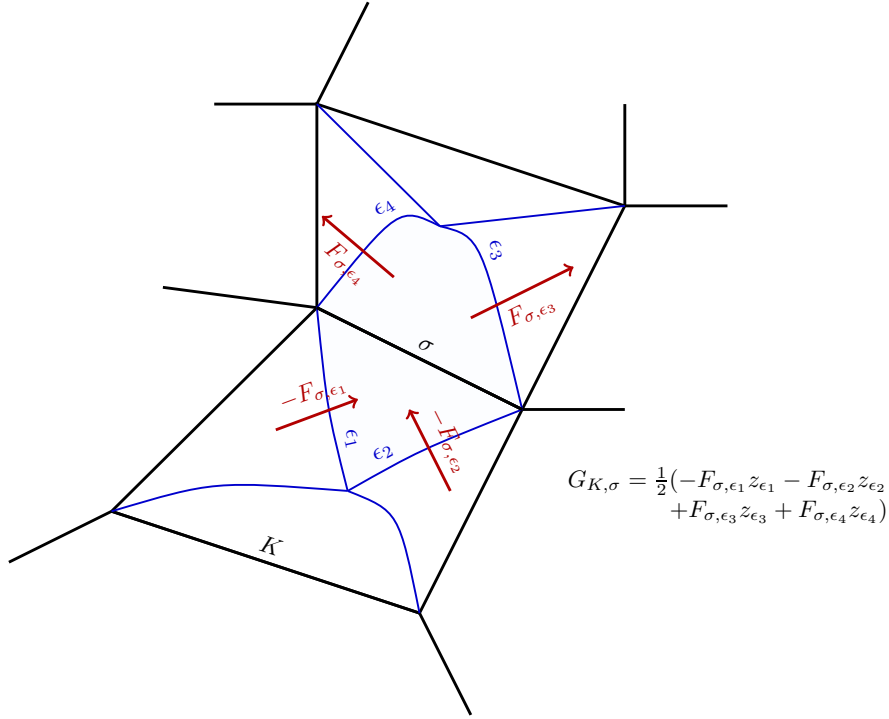
%
%
\section{Lax-Wendroff consistency of the primal mesh convection operator}\label{sec:lw}

The aim of the this section is to check the consistency, in the Lax-Wendroff sense, of the discrete operator ${\mathcal C}_K^n$.

\medskip
We suppose given a sequence of meshes $(\mesh\exm)_{m \in \xN}$ and time steps $(\delta t\exm)_{m \in \xN}$, with $h_{\mesh\exm}$ and $\delta t\exm$ tending to zero as $m$ tends to $+\infty$.
For $m \in \xN$, let us denote by $\rho\exm$, $\bfu\exm$ and $z\exm$ the discrete functions corresponding to the approximation on the mesh $\mesh\exm$ and with the time step $\delta t\exm$ of $\rho$, $\bfu$ and $z$ respectively, defined by:
\begin{align}
\label{eq:disc_rho} &
\rho\exm(\bfx,t)=\sum_{n=0}^{N\exm-1} \sum_{K\in\mesh\exm} \rho_K^n \ \mathcal X_K\ \mathcal X_{[t_n,t_{n+1})},
\\ 
\label{eq:disc_u} &
\bfu\exm(\bfx,t)=\sum_{n=0}^{N\exm-1} \sum_{\edge\in\edges\exm} \bfu_\edge^n \ \mathcal X_{D_\edge}\ \mathcal X_{[t_n,t_{n+1})},
\\ 
\label{eq:disc_z} &
z\exm(\bfx,t)=\sum_{n=0}^{N\exm-1} \sum_{\edge\in\edges\exm} z_\edge^n \ \mathcal X_{D_\edge}\ \mathcal X_{[t_n,t_{n+1})},
\end{align}
where $\mathcal X_K$, $\mathcal X_{D_\edge}$ and $\mathcal X_{[t_n,t_{n+1})}$ stand for the characteristic function of $K$, $D_\edge$ and the interval $[t_n,t_{n+1})$ respectively.

\medskip
We denote by $\theta_{\mesh,1}$ the following measure of the regularity of a mesh:
\begin{equation} \label{eq:theta1}
\theta_{\mesh,1} = \max \bigl\{\frac{|K|}{|L|}, \frac{|L|}{|K|} \mbox{ for } \edge=K|L \in \edgesint,  \bigr\},
\end{equation}
Note that, since $|D_{K,\edge}|=|K|/\zeta$ (recall that $\zeta$ is the number of faces of $K$), we have
\begin{equation}
|D_\edge|=\frac 1 \zeta \ (|K|+|L|) \leq \frac 1 \zeta\ (1+\theta_{\mesh,1})\ |K|.
\label{eq:majDedge}
\end{equation}

\medskip
Thanks to \eqref{eq:majDedge}, the proof of the following lemma is a straightforward consequence of \cite[Lemma 4.3]{gal-19-wea}.

\medskip
\begin{lemma}[Limit of the difference of unknowns translates]\label{lmm:trans}
Let $\theta >0$ and $(\mesh\exm)_{m \in \xN}$ be a sequence of meshes such that $\theta_{\mesh\exm,1} \le \theta$ for all $m \in \xN$ and $\lim_{m \to +\infty} h_{\mesh\exm} =0$.
We also assume that $\delta t\exm$ tends to zero.

Let $(\varrho\exm)_{m\in \xN}$ and $(z\exm)_{m\in \xN}$ be a sequence of discrete functions associated to the sequence of meshes and time steps, that are piecewise constant over respectively the primal cells and the dual cells.
Let us suppose that $(\varrho\exm)_{m\in \xN}$ and $(z\exm)_{m\in \xN}$ converge in $L^1(\Omega\times(0,T))$ to $\bar \varrho$ and $\bar z$ respectively.
Then
\begin{equation}\label{eq:trans_prim}
\lim_{m \to +\infty} \sum_{n=0}^{N\exm-1} \delta t\exm \sum_{\substack{\edge\in\edges\exm\\ \edge=K|L}} |D_\edge|\ |\varrho_K^n - \varrho_L^n| =0,
\end{equation}
and
\begin{equation}\label{eq:trans_dual}
\lim_{m \to +\infty} \sum_{n=0}^{N\exm-1} \delta t\exm \sum_{K \in\mesh\exm} |K|\sum_{\edge,\edgeprim \in \edges(K)} |z_\edge^n-z_{\edgeprim}^n| =0.
\end{equation}
\end{lemma}

Let $\varphi \in C^\infty_c(\Omega\times[0,T))$ and let us define $\varphi_K^n$ by
\begin{equation}
\varphi_K^n = \varphi(\bfx_K,t_n), \mbox{ for } K \in \mesh\exm \mbox{ and } 0 \leq n \leq N\exm,
\label{eq:defphiKn}
\end{equation}
where $\bfx_K$ stands for an arbitrary point of $K$.
For $\edge \in \edgesint$, $\edge =K|L$, we denote by $d_\edge$ the distance $d_\edge=|\bfx_K-\bfx_L|$, and we introduce the following additional measure of the mesh regularity:
\begin{equation}\label{eq:theta2}
	\theta_{\mesh,2}=\max_{K\in \mesh} \Bigl\{\frac {\sum_{\edgeprim\in\edges(K)} |\edgeprim|} {|K|}\ \max_{\edge = K|L \in \edges(K)} \{h_K+h_L\} \Bigr\}.
\end{equation}
Note that for $\edge \in \edgesint$, $\edge=K|L$, we have
\begin{equation}\label{eq:theta2_cor}
|\edge|\ d_\edge \leq |\edge|\ (h_K+h_L) \leq \theta_{\mesh,2} \min(|K|,|L|) \leq \frac{2 \theta_{\mesh,2}} \zeta\ |D_\edge|.
\end{equation}
Also note that the parameter $\theta_{\mesh,2}$ defined in \eqref{eq:theta2} is bounded independently of the mesh size if there exists $C\in \xR_+$ such that $|K| \ge Ch$ and $|\edge |\le Ch^{d-1}$, for all $K \in \mesh$ and $\edge \in \edges$.

Let the piecewise vector function $\gradi_{\edges,\tdisc}\varphi$ be defined by:
\begin{multline} \label{eq:gradT}
\gradi_{\edges,\tdisc} \varphi = \sum_{n=0}^{N-1}\ \sum_{\sigma \in \edges} (\gradi_{\edges,\tdisc}\varphi)_\edge^n\ \mathcal{X}_{D_\sigma}(\bfx)\ \mathcal{X}_{[t_n,t_{n+1}[}(t), \mbox{ with }
\\
(\gradi_{\edges,\tdisc} \varphi)_\edge^n =
 \begin{cases} \displaystyle
\frac{|\edge|}{|D_\edge|}\ (\varphi_L^n-\varphi_K^n)\, \bfn_{K,\edge} & \mbox{ if }\edge\in\edgesint, \edge=K|L,
\\[2ex] \displaystyle
0 & \mbox{ if }\edge\in\edgesext.
\end{cases}
\end{multline}
This vector function may be seen as a non-consistent discrete gradient which seems to appear first in \cite{eym-00-hconv}, where its convergence properties are shown for specific meshes and norms. 
The following weak convergence result holds \cite[Theorem 3.2]{gal-19-wea}.

\medskip
\begin{lemma}\label{lmm:gradw}
Let $(\mesh\exm)_{m \in \xN}$ be a sequence of meshes such that the mesh step  $h_{\mesh\exm}$ tends to zero when $m$ tends to $+\infty$.
We suppose that there exists a real number $\theta$ such that $\theta_{\mesh\exm,2} \leq \theta$ for $m\in\xN$.
Let $(\delta t\exm)_{m\in\xN}$ be a sequence of time steps such that $\delta t\exm$ tends to zero when $m$ tends to $+\infty$.
Let $\varphi \in C_c^\infty(\Omega \times [0,T))$ and, for $m \in \xN$, $\gradi_{\edges\exm,\tdisc\exm}\, \varphi \in L^\infty(\Omega \times (0,T))^d$ be defined by \eqref{eq:gradT}.

Then the sequence $(\gradi_{\edges\exm,\tdisc\exm}\, \varphi)_{m \in \xN}$ is bounded in $L^\infty(\Omega \times (0,T))^d$ uniformly with respect to $m$ and converges to $\gradi \varphi$ in $L^\infty(\Omega \times (0,T))^d$ weak $\star$.
\end{lemma}

\medskip
Multiplying \eqref{eq:defrhoz} by $\varphi_K^n$ defined by \eqref{eq:defphiKn}, summing over $K \in \mesh\exm$ and $n= 0, \ldots, N\exm-1$ yields 
\begin{equation}\label{eq:weak_conv}
\sum_{n=0}^{N\exm-1} \delta t \sum_{K\in\mesh\exm} {\mathcal C}_K^n z\ \varphi_K^n = T_{\partial t}\exm + T_{\rm div}\exm
\end{equation}
with
\begin{align}\label{eq:tint} &
T_{\partial t}\exm=\sum_{n=0}^{N\exm-1} \sum_{K\in\mesh\exm} |K|\ \bigl( (\rho z)_K^{n+1} - (\rho z)_K^n\bigr) \ \varphi_K^n,
\\ \label{eq:tdiv} &
T_{\rm div}\exm= \sum_{n=0}^{N\exm-1} \delta t\exm  \sum_{K\in\mesh\exm} \varphi_K^n \sum_{\edge \in \edges(K)} G_{K,\edge}^n.
\end{align}
We are now in position to state the following consistency results.

\medskip
\begin{lemma}[Consistency of the time derivative]\label{lmm:dt_term}
Let $(\mesh\exm)_{m \in \xN}$ and $(\delta t\exm)_{m\in\xN}$ be a sequence of meshes and time steps such that the space and time steps, $h_{\mesh\exm}$ and $\delta t\exm$, tend to zero when $m$ tends to $+\infty$.
Let $(\rho\exm)_{m\in \xN}$ and $(z\exm)_{m\in \xN}$ be a sequence of discrete functions associated to $(\mesh\exm)_{m \in \xN}$ and $(\delta t\exm)_{m\in\xN}$.
Let us suppose that $(\rho\exm)_{m\in \xN}$ and $(z\exm)_{m\in \xN}$ converge in $L^1(\Omega\times(0,T))$ to $\bar \rho$ and $\bar z$ respectively, and that both sequences are bounded in $L^\infty(\Omega\times(0,T))$.
Finally, let $\varphi \in C^\infty_c(\Omega \times [0,T)$ and let $T_{\partial t}\exm$ be defined by \eqref{eq:tint}.
Then
\[
\lim_{m \to +\infty} T_{\partial t}\exm = -\int_0^T \int_\Omega \bar \rho\, \bar z\ \partial_t \varphi \dx \dt - \int_\Omega \rho_0\,z_0\,\varphi(\bfx,0)\dx.
\]
\end{lemma}

\begin{proof}
Let $(\rho z)\exm$ be the function associated to $((\rho z)_K^n)_{K\in \mesh, 0\leq n <N})$ that is piecewise constant on the primal cells and time intervals.
Since in general, $|D_{K,\edge}| \neq \frac 1 2\, |D_\edge|$ (so that $|K| \neq \frac 1 2 \sum_{\edge\in\edges(K)} |D_\edge|$), the quantity $(\rho z)_K$ given by Equation \eqref{eq:defrhoz} is not a convex combination of the value of the density and $z$ at the primal faces, so that it is not clear that $(\rho z)\exm$ converges to $\bar \rho \bar z$ in $L^1(\Omega\times(0,T)$ as $m\to +\infty$.
However, since $\{D_\edge, \edge \in \edges\}$ is a partition of $\Omega$, we may expect that $(\rho z)\exm$ weakly converges to $\bar \rho \bar z$; we  prove this result in a first preliminary step and prove the consistency of the time derivative term in a second step.

\smallskip
{\bf Step 1. Weak convergence of $(\rho z)\exm$.} Let $\psi \in C^\infty_c(\Omega\times[0,T))$, and, for $K \in \mesh\exm$ and $0 \leq n \leq N\exm$, let us define $\widetilde \psi_K^n$ as the mean value of $\psi$ over $K \times (t_n,t_{n+1})$.
Let us write
\[
I\exm = \int_0^T\int_\Omega (\rho z)\exm \psi \dx \dt=
\sum_{n=0}^{N\exm-1} \delta t \sum_{K\in\mesh\exm} |K|\ (\rho z)_K^n \widetilde \psi_K^n.
\]
By the definitions \eqref{eq:defrhoz} and \eqref{eq:rhoface}, we get:
\[
I\exm=\sum_{n=0}^{N\exm-1} \delta t \sum_{K\in\mesh\exm} \sum_{\edge\in\edges(K)} \frac1 2
\ \bigl(|D_{K,\edge}|\ \rho_K^n + |D_{L,\edge}|\ \rho_L^n \bigr)\ z_\edge^n\ \widetilde \psi_K^n.
\]
We now decompose $I\exm=\mathcal{T}_1\exm + R_1\exm$ with
\[
\mathcal{T}_1\exm =\sum_{n=0}^{N\exm-1} \delta t \sum_{K\in\mesh\exm} \sum_{\edge\in\edges(K)}
\ |D_{K,\edge}|\ \rho_K^n\ z_\edge^n\ \widetilde\psi_K^n.
\]
Then
\[
\mathcal{T}_1\exm =\int_0^T \int_\Omega \rho\exm z\exm \psi \dx \dt.
\]
The sequences $(\rho\exm)_{m \in \xN}$ and $(z\exm)_{m \in \xN}$ converge to $\bar \rho$ and $\bar z$ in $L^p(\Omega\times(0,T))$ for $1\leq p <+\infty$ (since, by assumption, these sequences converge in $L^1(\Omega\times(0,T))$ and are uniformly bounded).
We thus have:
\[
\lim_{m \to +\infty}\mathcal{T}_1\exm =\int_0^T \int_\Omega \bar \rho\ \bar z\ \psi \dx \dt.
\]
The remainder term $R_2\exm$ reads
\[
R_1\exm = \sum_{n=0}^{N\exm-1} \delta t \sum_{K\in\mesh\exm} \sum_{\edge\in\edges(K)} \frac1 2
\ \bigl(|D_{L,\edge}|\ \rho_L^n - |D_{K,\edge}|\ \rho_K^n\bigr)\ z_\edge^n\ \widetilde \psi_K^n.
\]
For $m$ large enough, since the support in space of $\psi$ is compact in $\Omega$, $\psi$ vanishes in the cells having a face included in the boundary, and a reordering of the sums yields
\begin{align*}
R_1\exm &= \sum_{n=0}^{N\exm-1} \delta t \sum_{K\in\mesh\exm} \sum_{\substack{\edge \in \edges(K)\\\edge=K|L}} \frac 1 2 |D_{K,\edge}|\ \rho_K^n\ z_\edge^n\ (\widetilde \psi_L^n - \widetilde \psi_K^n), \\
	&\le   C_\psi\  T \ |\Omega|\ ||\rho\exm||_{L^\infty(\Omega\times(0,T))}\ ||z\exm||_{L^\infty(\Omega\times(0,T))}\ h_{\mesh\exm}
\end{align*}
where $C_\psi \in \xR_+$ depends only on $\psi$;
therefore, $R_1\exm$ tends to zero when $m$ tends to $+\infty$.
Finally, we observe that the sequence $(\rho z)_{m \in \xN}$ is bounded in $L^\infty(\Omega\times(0,T))$, so that, by density, we obtain the weak convergence of this sequence in $L^p(\Omega\times(0,T))$, for $1 \leq p < +\infty$ (note that weak convergence for a limited range of indexes $p$ would have been obtained under weaker assumptions on the unknowns, namely a control in $L^q$ with a given finite $q$, only under regularity assumptions for the mesh; uniform boundedness  is the only case where no such assumption is needed).
A similar proof shows that the function of space only $(\rho z)^0$ obtained from the initial conditions \eqref{eq:inicond} weakly converges to $\rho_0 z_0$ in $L^p(\Omega)$, for $1 \leq p < +\infty$.

\smallskip
{\bf Step 2. Weak consistency of the time term.}
Let us now turn to $T_{\partial t}\exm$.
A discrete integration by parts with respect to time yields $T_{\partial t}\exm=T_{\partial t,1}\exm+T_{\partial t,2}\exm$ with
\begin{align*} 
&T_{\partial t,1}\exm=-\sum_{n=0}^{N\exm-1} \delta t \sum_{K\in\mesh\exm} |K|\ (\rho z)_K^{n+1}\ \frac{\varphi_K^{n+1}-\varphi_K^n}{\delta t},
\\
&T_{\partial t,2}\exm=-\sum_{K\in\mesh\exm} |K| (\rho z)_K^0\ \varphi_K^0.
\end{align*}
Let us denote by $\eth_t\exm \varphi$ the following time discrete derivative of $\varphi$:
\[
\eth_t\exm \varphi = \sum_{n=0}^{N\exm-1} \sum_{K\in\mesh\exm} \frac{\varphi_K^{n+1}-\varphi_K^n}{\delta t}\ \mathcal X_K\ \mathcal X_{[t_n,t_{n+1})}.
\]
With this notation, we get:
\[
T_{\partial t,1}\exm=-\int_{\delta t}^T \int_\Omega (\rho z)\exm \eth_t\exm \varphi(.,t-\delta t)\dx \dt.
\]
Thanks to the regularity of $\varphi$, $\eth_t\exm \varphi$ tends to $\partial_t \varphi$ when $m$ tends to $+\infty$ in $L^\infty(\Omega\times(0,T))$, so that, thanks to the weak convergence of the sequence $((\rho z)\exm)_{m\in\xN}$,
\[
\lim_{m \to +\infty} T_{\partial t,1}\exm = -\int_0^T \int_\Omega \bar \rho\ \bar z\ \partial_t \varphi \dx \dt.
\]
Similarly, the sequence of piecewise functions of space $((\varphi\exm)^0)_{m\in\xN}$ converges to $\varphi(.,0)$  in $L^\infty(\Omega)$, so that
\[
\lim_{m \to +\infty} T_{\partial t,2}\exm = -\int_\Omega \rho_0\,z_0\ \varphi(\bfx,0) \dx,
\]
which concludes the proof.
\end{proof}
%
%

\medskip
We now turn to the second term of \eqref{eq:weak_conv}, namely the divergence term $T_{\rm div}\exm$.

\medskip
\begin{lemma}[Consistency of the divergence term] \label{lmm:div_term}
Let $\theta >0$  and $(\mesh\exm)_{m \in \xN}$ be a sequence of meshes such that $\max(\theta_{\mesh\exm,1},\theta_{\mesh\exm,2}) \le \theta$ for all $m \in \xN$ and $\lim_{m \to +\infty} h_{\mesh\exm} =0$.
Let $(\delta t\exm)_{m\in\xN}$ be a sequence of time steps such that $\delta t\exm$ tends to zero when $m$ tends to $+\infty$.
Let $(\rho\exm)_{m\in \xN}$, $(\bfu\exm)_{m\in \xN}$ and $(z\exm)_{m\in \xN}$ be a sequence of discrete functions associated to to $(\mesh\exm)_{m \in \xN}$ and $(\delta t\exm)_{m\in\xN}$.
Let us suppose that $(\rho\exm)_{m\in \xN}$ and $(z\exm)_{m\in \xN}$ converge in $L^1(\Omega\times(0,T))$ to $\bar \rho$ and $\bar z$ respectively, and that both sequences are bounded in $L^\infty(\Omega\times(0,T))$.
We assume in addition that $(\bfu\exm)_{m\in \xN}$ converges to $\bar \bfu$ in $L^1(\Omega\times(0,T))^d$, and is bounded in $L^\infty(\Omega\times(0,T))^d$.
Finally, let $\varphi \in C^\infty_c(\Omega \times [0,T)$ and let $T_{\rm div}\exm$ be defined by \eqref{eq:tdiv}.
Then

\[
\lim_{m \to +\infty} T_{\rm div}\exm = -\int_0^T \int_\Omega \bar \rho\, \bar z\, \bar\bfu \cdot \gradi \varphi \dx \dt.
\]
\end{lemma}

\medskip
\begin{proof}
Since the support in space of $\varphi$ is compact in $\Omega$, for $m$ large enough, $\varphi_K^n$ vanishes for $0\leq n \leq N\exm-1$ for every cell $K$ having one of its face on the boundary.
For such an index $m$, reordering the summations, we get:
\begin{equation}\label{eq:tdiv2}
T_{\rm div}\exm = \sum_{n=0}^{N\exm-1} \delta t\exm  \sum_{\substack{\edge\in\edgesint\exm\\ \edge=K|L}}  G_{K,\edge}^n\ (\varphi_K^n-\varphi_L^n).
\end{equation}
Let $\edge \in \edgesint$, $\edge=K|L$.
By the property \eqref{eq:mass_hdiam} of the mass fluxes through the dual cells, we get
\[
\begin{array}{ll} \displaystyle
-\sum_{\substack{\edged \in \edges(D_\edge)\\ \edged \subset K}} F_{\edge,\edged}^n\ z_\edged^n
& \displaystyle
= -z_\edge^n \sum_{\substack{\edged \in \edges(D_\edge)\\ \edged \subset K}} F_{\edge,\edged}^n + R_{\edge,1}^n
\\  & \displaystyle
= z_\edge^n\ \Bigl( F_{K,\edge}^n - \frac 1 \zeta \sum_{\edgeprim\in \edges(K)} F_{K,\edgeprim}^n \Bigr) + R_{\edge,1}^n,
\end{array}
\]
with
\[
R_{\edge,1}^n = -\sum_{\substack{\edged \in \edges(D_\edge)\\ \edged \subset K}} F_{\edge,\edged}^n\ (z_\edged^n-z_\edge^n).
\]
Similarly:
\[
\sum_{\substack{\edged \in \edges(D_\edge)\\ \edged \not \subset K}} F_{\edge,\edged}^n\ z_\edged^n
=z_\edge^n \Bigl( F_{K,\edge}^n + \frac 1 \zeta \sum_{\edgeprim\in \edges(L)} F_{L,\edgeprim}^n \Bigr) + R_{\edge,2}^n,
\]
with
\[
R_{\edge,2}^n = \sum_{\substack{\edged \in \edges(D_\edge)\\ \edged \not \subset K}} F_{\edge,\edged}^n\ (z_\edged^n-z_\edge^n).
\]
Finally, we get
\begin{equation}\label{eq:G}
G_{K,\edge}^n= F_{K,\edge}^n\ z_\edge^n + \frac 1 2 R_{\edge,1}^n + \frac 1 2 R_{\edge,2}^n + R_{\edge,3}^n,
\end{equation}
with
\[
R_{\edge,3}^n 
= \frac 1 {2\zeta}\ z_\edge^n \Bigl(\sum_{\edgeprim\in \edges(L)} F_{L,\edgeprim}^n-\sum_{\edgeprim\in \edges(K)} F_{K,\edgeprim}^n\Bigr)
\]
Using the expression \eqref{eq:G} in \eqref{eq:tdiv2}, we get $T_{\rm div}\exm=T_0\exm+T_2\exm  +T_3\exm$, with
\[
\begin{array}{l} \displaystyle
T_0\exm =\sum_{n=0}^{N\exm-1} \delta t\exm  \sum_{\substack{\edge\in\edgesint\exm \\ \edge=K|L}} F_{K,\edge}^n\ z_\edge^n\ (\varphi_K^n-\varphi_L^n),
\\ \displaystyle
T_2\exm = \frac 1 2 \sum_{n=0}^{N\exm-1} \delta t\exm  \sum_{\substack{\edge\in\edgesint\exm\\\edge=K|L}} (R_{\edge,1}^n+R_{\edge,2}^n)\ (\varphi_K^n-\varphi_L^n),
\\ \displaystyle
T_3\exm = \sum_{n=0}^{N\exm-1} \delta t\exm  \sum_{\substack{\edge\in\edgesint\exm\\ \edge=K|L}} R_{\edge,3}^n\ (\varphi_K^n-\varphi_L^n).
\end{array}
\]
We begin with the term $T_0\exm$, which we decompose as $T_0\exm=T_1\exm + R_0\exm$ with:
\begin{align*}
T_1\exm & = -\sum_{n=0}^{N\exm-1} \delta t\exm  \sum_{\substack{\edge\in\edgesint\exm \\ \edge=K|L}}
 (|D_{K,\edge}|\ \rho_K^n + |D_{L,\edge}|\ \rho_L^n)\ z_\edge^n\  \bfu_\edge^n \cdot \frac{|\edge|}{|D_\edge|}(\varphi_L^n-\varphi_K^n)\,\bfn_{K,\edge}
\\
& = -\int_0^T \int_\Omega \rho\exm\, z\exm\, \bfu\exm \cdot \gradi_{\edges\exm,\tdisc\exm} \varphi \dx \dt;
\end{align*}
again using the fact that  the convergence in $L^1(\Omega\times(0,T))$ and the boundedness in $L^\infty(\Omega\times(0,T))$ implies the convergence in any $L^p(\Omega\times(0,T))$, $1\le p<+\infty$, we get by Lemma \ref{lmm:gradw} that
\[
\lim_{m \to +\infty} T_1\exm = -\int_0^T \int_\Omega \bar \rho\, \bar z\, \bar \bfu \cdot \gradi \varphi \dx \dt.
\]
The residual term $R_0\exm$ reads:
\begin{multline*}
R_0\exm=\sum_{n=0}^{N\exm-1} \delta t\exm  \sum_{\substack{\edge\in\edgesint\exm\\ \edge=K|L}} |\edge|
\ \Bigl[\frac{|D_{K,\edge}|}{|D_\edge|}\ (\rho_\edge^n-\rho_K^n) + \frac{|D_{L,\edge}|}{|D_\edge|}\ (\rho_\edge^n-\rho_L^n)\Bigr]
\\
\ z_\edge^n\  \bfu_\edge^n \cdot \bfn_{K,\edge}\ (\varphi_K^n-\varphi_L^n),
\end{multline*}
The regularity of $\varphi$ implies that, for $\edge \in \edges$, $\edge=K|L$ and $0 \leq n \leq N-1$, $|\varphi_K^n-\varphi_L^n| \leq C_\varphi d_\edge$.
Hence, since $\rho_\edge^n$ is a convex combination of $\rho_K^n$ and $\rho_L^n$, there exists $C >0$, depending on the $L^\infty$ bound of the unknowns and on $\varphi$, such that
\[
|R_0\exm| \leq C \sum_{n=0}^{N\exm-1} \delta t\exm  \sum_{\substack{\edge\in\edgesint\exm\\ \edge=K|L}} |\edge|\ d_\edge\ |\rho_K^n-\rho_L^n|,
\]
and thus $R_0\exm$ tends to zero when $m$ tends to $+\infty$ thanks to the assumed regularity of the sequence of meshes and Lemma~\ref{lmm:trans}.

\medskip
The term $T_2\exm$ reads:
\begin{multline*}
T_2\exm=\frac 1 2 \!\!\sum_{n=0}^{N\exm-1} \delta t\exm  \!\!\!\sum_{\substack{\edge\in\edges\exm\\ \edge=K|L}}
\Bigl[\!
\sum_{\substack{\edged \in \edges(D_\edge)\\ \edged \subset L}}\!\! F_{\edge,\edged}^n\ (z_\edged^n-z_\edge^n)
-\!\!\sum_{\substack{\edged \in \edges(D_\edge) \\ \edged \subset K}} \!\!F_{\edge,\edged}^n\ (z_\edged^n-z_\edge^n)
\Bigr]
\ (\varphi_K^n-\varphi_L^n).
\end{multline*}
For $\edged=\edge|\edgeprim$, $z_\edged$ is supposed to be a convex combination of $z_\edge$ and $z_{\edgeprim}$, so that $|z_\edged-z_\edge| \leq |z_{\edgeprim}-z_\edge|$.
We thus get, reordering the sums, with the same constant $C_\varphi$ as previously:
\[
|T_2\exm| \leq  C_\varphi \sum_{n=0}^{N\exm-1} \delta t\exm \sum_{K\in\mesh\exm} \sum_{\substack{\edged \subset K\\ \edged=\edge|\edgeprim}}
(d_\edge + d_{\edgeprim})\ |F_{\edge,\edged}^n|\ |z_\edge^n -z_{\edgeprim}^n|
\]
Using now the expressions \eqref{eq:fluxd} and \eqref{primal-mass-flux}, and introducing $\xi =\max_{\edgeprim \in \edges(K)} \xi_{\epsilon,\tau}$ (recall that the coefficients $\xi_{\epsilon,\tau}$ do not depend on the cell $K$), we get:
\begin{multline*}
|T_2\exm| \leq  C_\varphi\,\xi \!\!\sum_{n=0}^{N\exm-1}\!\! \!\delta t\exm \!\!\!\sum_{K\in\mesh\exm}\!\! |K| \sum_{\substack{\edged \subset K\\ \edged=\edge|\edgeprim}}
\frac 1 {|K|} (d_\edge + d_{\edgeprim})
\!\Bigl(\!\! \sum_{\omega \in \edges(K)} \!\!|\omega|\ \rho_{\omega}^n\ |\bfu_{\omega}^n| \Bigr)  |z_\edge^n -z_{\edgeprim}^n|.
\end{multline*}
Since the unknowns are assumed to be bounded, by the assumed regularity of the sequence of meshes (see \eqref{eq:theta2} for the definition of the regularity parameter $\theta_{\mesh,2}$), Lemma \ref{lmm:trans} thus implies that $T_2\exm$ tends to zero when $m$ tends to $+\infty$.

\medskip
Let us turn to the term $T_3\exm$; using the fact that, for any cell $K$, 
\[\sum_{\edge \in \edges(K)} |\edge|\ \bfn_{K,\edge}=0,\] 
and denoting by $\bfu^n_K$ a convex combination of the face velocities $(\bfu_\edge)_{\edge \in \edges(K)}$, we get:
\[
\begin{array}{rl} \displaystyle
T_3\exm 
&= \displaystyle \sum_{n=0}^{N\exm-1} \delta t\exm \!\!\! \sum_{\substack{\edge\in\edgesint\exm \\ \edge=K|L}}
\frac 1 {2\zeta}\ z_\edge^n \Bigl(\sum_{\edgeprim\in \edges(L)} F_{L,\edgeprim}^n-\sum_{\edgeprim\in \edges(K)} F_{K,\edgeprim}^n\Bigr)(\varphi_K^n-\varphi_L^n)
\\[3ex]  
&= \displaystyle \sum_{n=0}^{N\exm-1} \delta t\exm \!\!\!  \sum_{\substack{\edge\in\edgesint\exm\\ \edge=K|L}}
\frac 1 {2\zeta}\ z_\edge^n \Bigl(\sum_{\edgeprim\in \edges(L)} |\edgeprim|\ (\rho_{\edgeprim}^n \bfu_{\edgeprim}^n - \rho_L^n \bfu_L^n) \cdot \bfn_{L,\edgeprim} 
\\ 
&  \hfill
- \displaystyle \sum_{\edgeprim\in \edges(K)} |\edgeprim|\ (\rho_{\edgeprim}^n \bfu_{\edgeprim}^n - \rho_K^n \bfu_K^n)\cdot \bfn_{K,\edgeprim}\Bigr)(\varphi_K^n-\varphi_L^n).
\end{array}
\]
Using the regularity of $\varphi$ and the fact that the unknowns are assumed to be bounded, we get by a reordering of the summations that there exists $C_\varphi$ only depending on $\varphi$ such that
\[
|T_3\exm| \leq C_\varphi \sum_{n=0}^{N\exm-1} \delta t\exm  \sum_{\substack{\edge\in\edges\exm\\ \edge=K|L}}
|V_{\edge,K}|\ |\rho_\edge^n \bfu_\edge^n-\rho_K^n \bfu_K^n| + |V_{\edge,L}|\ |\rho_\edge^n \bfu_\edge^n-\rho_L^n \bfu_L^n|,
\]
where the quantity $|V_{\edge,K}|$ is homogeneous to a volume, reads
\[
|V_{\edge,K}| = |\edge| \sum_{\edgeprim \in \edges(K)} d_{\edgeprim}.
\]
and satisfies
\[
|V_{\edge,K}| < \sum_{\edge \in \edges(K)} |V_{\edge,K}|=\sum_{\edge,\,\edgeprim \in \edges(K)}|\edge|\ d_{\edgeprim}
= \sum_{\edgeprim \in \edges(K)} d_{\edgeprim} \sum_{\edge \in \edges(K)} |\edge| \leq \zeta \,\theta_{\mesh\exm,2}\ |K|.
\]
By the regularity assumption of the sequence of meshes (see the Definition \eqref{eq:theta1} of the regularity parameter $\theta_{\mesh,1}$), we also have $|V_{\edge,K}| < C |D_\edge|$ with $C$ independent of the mesh.
Developing the difference of products in this relation by the identity $2(ab-cd)=(a+c)(b-d)+(a-c)(b+d)$ for $(a,b,c,d)\in\xR^4$, invoking the assumed boundedness of the unknowns and Lemma \ref{lmm:trans} thus yields that $T_3\exm$ tends to zero when $m$ tends to $+\infty$.
\end{proof}
%
%
\section{Lax-Wendroff consistency of the scheme}\label{sec:lw_mom}

The consistency of the mass balance equation easily follows by the arguments developed in \cite{gal-19-wea}, and we focus here on the equations involving the velocity convection operator, namely the momentum and the total energy balance equations.

\subsection{The momentum balance equation}

We show in this section that the limit of a convergent sequence of discrete solutions satifies a weak form of the momentum balance equation.
To this purpose, we derive a discrete momentum balance equation posed over the primal cells, involving the reconstructed convection operator defined by \eqref{eq:conv_p_z}, and pass to the limit in the resulting relation.
The consistency of the convection operator is already treated in the previous section, and we essentially have to check the consistency of the obtained (primal cell) discrete pressure gradient to obtain the following result.

\medskip
\begin{theorem}[Consistency of the primal cell momentum balance equation] \label{theo:mom}
Let $\theta >0$  and $(\mesh\exm)_{m \in \xN}$ be a sequence of meshes such that $\max(\theta_{\mesh\exm,1},\theta_{\mesh\exm,2}) \le \theta$ for all $m \in \xN$ and $\lim_{m \to +\infty} h_{\mesh\exm} =0$.
We also assume that $\delta t\exm$ tends to zero.\\[1ex]
Let us suppose that $(\rho\exm)_{m\in \xN}$ and $(e\exm)_{m\in \xN}$ converge in $L^1(\Omega\times(0,T))$ to $\bar \rho$ and $\bar e$ respectively, and are bounded in $L^\infty(\Omega\times(0,T))$; let the sequence $(\bfu\exm)_{m\in \xN}$ converge in $L^1(\Omega\times(0,T))^d$ to $\bar \bfu$ and be bounded in $L^\infty(\Omega\times(0,T))^d$.\\[1ex]
Then the sequence $(p\exm)_{m\in \xN}$ defined by $p\exm = (\gamma-1)\,\rho\exm\,e\exm$ for $m \in \xN$ converges in $L^1(\Omega\times(0,T))$ to $\bar p = (\gamma-1)\,\bar \rho\,\bar e$  and the limits $\bar \rho$, $\bar \bfu$ and $\bar p$ satisfy the following weak form of the momentum balance equation which reads:
\begin{equation}\label{eq:mom_weak}
\int_0^T \int_\Omega \Bigl( \bar \rho \, \bar u_i\, \partial_t \varphi +  \bar \rho\, \bar u_i\, \bar \bfu \cdot \gradi \varphi + \bar p\, \partial_i \varphi \Bigr) \dx \dt
+ \int_\Omega \rho_0\ (\bfu_0)_i\ \varphi(\bfx,0)\dx=0,
\end{equation}
for any function $\varphi \in C^\infty_c(\Omega \times [0,T))$ and for $1 \leq i \leq d$.
\end{theorem}

\medskip
\begin{proof}
Since the sequences $(\rho\exm)_{m\in \xN}$ and $(e\exm)_{m\in \xN}$ are uniformly bounded, their convergence in $L^1(\Omega\times(0,T))$ implies their convergence in $L^p(\Omega\times(0,T))$ for any finite $p >1$, which yields the convergence of $(p\exm)_{m\in \xN}$ to $\bar p$.
Let us now, for $m\in\xN$ and for $K \in \mesh\exm$, sum the discrete momentum balance \eqref{eq:scheme_mom} over the faces of $K$ and divide by 2.
We obtain, for $K\in\mesh\exm$, $0 \leq n < N\exm$ and $1 \leq i \leq d$, dropping some superscripts $\exm$ for short:
\begin{equation}\label{eq:momK}
{\mathcal C}_K^n u_i + \frac 1 2 {\substack{\edge \in \edges(K)\\\edge=K|L}}  |\edge|\ (p^{n+1}_L-p^{n+1}_K)\,\nkedgei = 0. 
\end{equation}
Let $\varphi \in C^\infty_c(\Omega \times [0,T)$ and $\varphi_K^n = \varphi(\bfx_K,t_n)$, for $K \in \mesh\exm$ and $0 \leq n \leq N\exm$, with $\bfx_K\in K$.
Multiplying \eqref{eq:momK} by $\delta t\exm \varphi_K^n$ and summing over the cells and time levels, we obtain
\[\begin{array}{l}\displaystyle
T_1\exm + T_2\exm=0, \mbox{ with }
\\[1ex] \qquad \displaystyle
T_1\exm=\sum_{n=0}^{N\exm-1} \delta t\exm \sum_{K\in\mesh\exm} {\mathcal C}_K^n u_i\ \varphi_K^n,
\\[3ex] \qquad \displaystyle
\ T_2\exm=\sum_{n=0}^{N\exm-1} \delta t\exm \sum_{K\in\mesh\exm} \frac 1 2 \ \varphi_K^n {\substack{\edge \in \edges(K)\\\edge=K|L}}  |\edge|\ (p^{n+1}_L-p^{n+1}_K)\,\nkedgei.
\end{array}\]
By lemmas \ref{lmm:dt_term} and \ref{lmm:div_term}, we get:
\[
\lim_{m \to +\infty} T_1\exm = - \int_0^T \int_\Omega \Bigl( \bar \rho \, \bar u_i\, \partial_t \varphi +  \bar \rho\, \bar  u_i\, \bar\bfu \cdot \gradi \varphi \Bigr) \dx \dt
- \int_\Omega \rho_0\ (u_0)_i\ \varphi(\bfx,0)\dx.
\]
Reordering the sums in $T_2\exm$ yields
\[\begin{array}{ll}
T_2\exm
& \displaystyle
= \sum_{n=0}^{N\exm-1} \delta t\exm \sum_{\substack{\edge \in \edgesint\exm\\ \edge=K|L}} |\edge|\ (p^{n+1}_L-p^{n+1}_K)\, \frac{\varphi_K^n + \varphi_L^n} 2\,\nkedgei
\\ & \displaystyle
= -\sum_{n=0}^{N\exm-1} \delta t\exm \sum_{K\in\mesh\exm} p^{n+1}_K {\substack{\edge \in \edges(K)\\ \edge=K|L}}  |\edge|\ \frac{\varphi_K^n + \varphi_L^n} 2\,\nkedgei.
\end{array}\]
For $\edge \in \edgesint\exm$ and $0\leq n < N\exm$, let
\[
\varphi_\edge^n=\frac 1 {\delta t\exm\ |\edge|} \int_{t_n}^{t_{n+1}} \int_\edge \varphi \dedge \dt,\quad
\delta\varphi_\edge^n = \frac{\varphi_K^n + \varphi_L^n} 2 - \varphi_\edge^n.
\]
By regularity of $\varphi$, there exists a real number $C_\varphi$ such that $|\delta\varphi_\edge^n| \leq C_\varphi\ (h_K + h_L)$.
With the notation
\[
(\partial_i \varphi)_K^n=\delta t\exm {\substack{\edge \in \edges(K)\\ \edge=K|L}}  |\edge|\ \,\nkedgei \varphi_\edge^n= \int_{t_n}^{t_{n+1}} \int_K \partial_i \varphi \dx \dt,
\]
we remark that
\begin{equation}
|(\partial_i \varphi)_K^n| \leq C_\varphi\ \delta t\exm\ |K|.
\label{eq:estim-diphi}
\end{equation}

We now split $T_2\exm$ as $T_2\exm=-T_{2,1}\exm-T_{2,2}\exm-T_{2,3}\exm$ with
\[\begin{array}{l} \displaystyle
T_{2,1}\exm = \sum_{n=0}^{N\exm-1} \delta t\exm \sum_{K\in\mesh\exm} p^n_K\ (\partial_i \varphi)_K^n,
\\[3ex] \displaystyle
T_{2,2}\exm= \sum_{n=0}^{N\exm-1} \delta t\exm \sum_{K\in\mesh\exm} (p^{n+1}_K-p_K^n)\ (\partial_i \varphi)_K^n,
\\[3ex] \displaystyle
T_{2,3}\exm=\sum_{n=0}^{N\exm-1} \delta t\exm \sum_{K\in\mesh\exm} p^{n+1}_K {\substack{\edge \in \edges(K)\\\edge=K|L}}  |\edge|\ \delta\varphi_\edge^n \nkedgei.
\end{array}\]
The first term reads
\[
T_{2,1}\exm = \int_0^T \int_\Omega p\exm\, \partial_i\varphi \dx \dt,
\]
so that, since $p\exm \to \bar p$ in in $L^1(\Omega\times(0,T))$ as $m \to +\infty$,
\[
\lim_{m \to +\infty} T_{2,1}\exm = \int_0^T \int_\Omega \bar p\, \partial_i\varphi \dx \dt.
\]
Thanks to \eqref{eq:estim-diphi}, we get for $T_{2,2}\exm$:
\[
|T_{2,2}\exm| \leq C_\varphi\ \sum_{n=0}^{N\exm-1} \delta t\exm \sum_{K\in\mesh\exm} |K|\ |p^{n+1}_K-p_K^n|,
\]
so $T_{2,2}\exm$ tends to zero when $m$ tends to $+\infty$ by \cite[Lemma 4.3]{gal-19-wea}.
Reordering the sums in $T_{2,3}\exm$, we get
\[
T_{2,3}\exm=\sum_{n=0}^{N\exm-1} \delta t\exm \sum_{\substack{\edge \in \edgesint\exm\\ \edge=K|L}} |\edge|\ \delta\varphi_\edge^n\ (p^{n+1}_K -p^{n+1}_L)\nkedgei.
\]
Thanks to the estimate of $|\delta\varphi_\edge^n|$ and the regularity assumption on the sequence of meshes which implies that $|\edge|\ (h_K +h_L) \leq C\ |D_\edge|$ (see Inequality \eqref{eq:theta2_cor}), we obtain:
\[
|T_{2,3}\exm| \leq C_\varphi \sum_{n=0}^{N\exm-1} \delta t\exm \sum_{\substack{\edge \in \edgesint\exm\\ \edge=K|L}} |D_\edge|\ |p^{n+1}_K -p^{n+1}_L|,
\]
and $T_{2,3}\exm$ tends to zero when $m$ tends to $+\infty$ once again by \cite[Lemma 4.3]{gal-19-wea}, which concludes the proof.
\end{proof}
%
%
\subsection{The energy balance equation}\label{sec:lw_Etot}
We now derive a kinetic energy balance equation associated to the primal mesh and add to the internal energy balance equation to obtain a discrete total energy balance, which shows that the scheme is, in some sense, locally conservative with respect to the total energy.

\medskip
\begin{lemma}[Local discrete total energy balance]
A solution to the scheme \eqref{eq:scheme} satisfies the following total energy balance equation, for $0 \leq n \leq N-1$ and $K \in \mesh$: 
\begin{multline}\label{eq:etot}
\frac{|K|}{\delta t}\ \Bigl[\bigl(\rho_K^{n+2} e_K^{n+2} + (\rho\, \Ekin)_K^{n+1} \bigr) - \bigl(\rho_K^{n+1} e_K^{n+1}+ (\rho\, \Ekin)^n \bigr) \Bigr]
\\
+ {\substack{\edge \in \edges(K)\\\edge=K|L}}  F_{K,\edge}^{n+1}\, e_\edge^{n+1} + (\Gkin)_\edge^n
+ {\substack{\edge \in \edges(K)\\\edge=K|L}}  |\edge|\ \frac {p_K^{n+1} + p_L^{n+1}} 2\ \bfu_\edge^{n+1} \cdot \bfn_{K,\edge}
=0,
\end{multline}
where the kinetic energy $\Ekin$ and its associated flux $\Gkin$ result from the construction of the convection operator on the primal mesh from the convection on the dual mesh described in Section \ref{sec:prim_conv}:
\begin{equation} \begin{array}{l} \displaystyle
|K|\ (\rho\, \Ekin)_K^n= \frac 1 4 \sum_{\edge \in \edges(K)} |D_\edge|\ \rho_{D_\edge}^n\,|\bfu_\edge^n|^2,
\\[3ex] \displaystyle
(\Gkin)_\edge^n= -\frac 1 2 \sum_{\substack{\edged \in \edges(D_\edge)\\ \edged \subset K}} F_{\edge,\edged}^n\ |\bfu_\edged^n|^2
+ \frac 1 2\sum_{\substack{\edged \in \edges(D_\edge)\\ \edged \not \subset K}} F_{\edge,\edged}^n\ |\bfu_\edged^n|^2.
\end{array} \end{equation}
\end{lemma}

\begin{proof}
For $K \in \mesh$, summing over $\edge \in \edges(K)$ and $i=1,\ldots,d$ the kinetic energy balance \eqref{eq:Ekin} divided by two and adding the internal energy balance \eqref{eq:scheme_Eint} of the scheme written at the time step $n+2$ yields Equation \eqref{eq:etot}
Thanks to its \textit{ad hoc} definition \eqref{eq:def_SK}, the corrective term $S_K^{n+1}$ exactly compensates with the sum of the residual terms $R_\edge^{n+1}$ defined by \eqref{sw:eq:def_R} to yield a conservative equation.
\end{proof}

\medskip
We are now in position to prove the following weak consistency result by passing to the limit in the scheme.

\medskip
\begin{theorem}[Consistency of the total energy balance equation]
Let $\theta >0$  and $(\mesh\exm)_{m \in \xN}$ be a sequence of meshes such that $\max(\theta_{\mesh\exm,1},\theta_{\mesh\exm,2}) \le \theta$ for all $m \in \xN$ and $\lim_{m \to +\infty} h_{\mesh\exm} =0$.
We also assume that $\delta t\exm$ tends to zero.\\[1ex]
Let us suppose that $(\rho\exm)_{m\in \xN}$ and $(e\exm)_{m\in \xN}$ converge in $L^1(\Omega\times(0,T))$ to $\bar \rho$ and $\bar e$ respectively, and are bounded in $L^\infty(\Omega\times(0,T))$; let the sequence $(\bfu\exm)_{m\in \xN}$ converge in $L^1(\Omega\times(0,T))^d$ to $\bar \bfu$ and be bounded in $L^\infty(\Omega\times(0,T))^d$.
We suppose in addition that $\rho_0$ and $e_0$ belong to $W^{1,1}(\Omega)$ and $\bfu_0$ belongs to $W^{1,1}(\Omega)^d$.\\[1ex]
Then the sequence $(p\exm)_{m\in \xN}$ defined by $p\exm = (\gamma-1)\,\rho\exm\,e\exm$ for $m \in \xN$ converges in $L^1(\Omega\times(0,T))$ to $\bar p = (\gamma-1)\,\bar \rho\,\bar e$  and the limits $\bar \rho$, $\bar \bfu$, $\bar p$ and $\bar e$ satisfy the following weak form of the total energy balance equation:
\begin{multline}\label{eq:Etot_weak}
\int_0^T \int_\Omega \Bigl( \bar \rho \bar E\, \partial_t \varphi +  (\bar \rho\, \bar E+\bar p)\, \bar \bfu \cdot \gradi \varphi \Bigr) \dx \dt
+ \int_\Omega \rho_0 \,(E)_0\ \varphi(\bfx,0)\dx=0,
\\
\mbox{with } \bar E = \frac 1 2 \, |\bar \bfu|^2 + \bar e,\ (E)_0 = \frac 1 2 \, |\bfu_0|^2 + e_0
\end{multline}
and for any function $\varphi \in C^\infty_c(\Omega \times [0,T))$.
\end{theorem}

\medskip
\begin{proof}
The convergence of $(p\exm)_{m\in \xN}$ to $\bar p$ is shown in the proof of Theorem \ref{theo:mom}.
We now pass to the limit in a weak form of Equation \eqref{eq:etot}.
Let $\varphi \in C^\infty_c(\Omega \times [0,T)$, $m\in\xN$ and $\varphi_K^n = \varphi(\bfx_K,t_n)$, for $K \in \mesh\exm$ and $0 \leq n \leq N\exm$, with $\bfx_K \in K$.
Multiplying \eqref{eq:etot} by $\delta t\exm \varphi_K^n$ and summing over the cells and time levels, we obtain, dropping the superscripts $\exm$ for short whenever it does not hinder comprehension,
\[\begin{array}{l}\displaystyle
T_1\exm + T_2\exm + T_3\exm + T_4\exm + T_5\exm=0, \mbox{ with }
\\[1ex] \qquad \displaystyle
T_1\exm=\sum_{n=0}^{N\exm-2} \sum_{K\in\mesh\exm} |K|\ (\rho_K^{n+2} e_K^{n+2} - \rho_K^{n+1} e_K^{n+1})\ \varphi_K^n,
\\[1ex] \qquad \displaystyle
T_2\exm=\sum_{n=0}^{N\exm-2} \sum_{K\in\mesh\exm} |K|\ \Bigl((\rho\, \Ekin)_K^{n+1} - (\rho\, \Ekin)_K^n \Bigr)\ \varphi_K^n,
\\[3ex] \qquad \displaystyle
\ T_3\exm=\sum_{n=0}^{N\exm-1} \delta t\exm \sum_{K\in\mesh\exm} \ \varphi_K^n \sum_{\edge \in \edges(K)}  |\edge|\ F_{K,\edge}^{n+1}\, e_\edge^{n+1},
\\[3ex] \qquad \displaystyle
\ T_4\exm=\sum_{n=0}^{N\exm-1} \delta t\exm \sum_{K\in\mesh\exm} \ \varphi_K^n \sum_{\edge \in \edges(K)} (\Gkin)_\edge^n,
\\[3ex] \qquad \displaystyle
\ T_5\exm=\sum_{n=0}^{N\exm-1} \delta t\exm \sum_{K\in\mesh\exm} \ \varphi_K^n \sum{\substack{\edge \in \edges(K)\\\edge=K|L}}  |\edge|\ \frac {p_K^{n+1} + p_L^{n+1}} 2\ \bfu_\edge^{n+1} \cdot \bfn_{K,\edge}.
\end{array}\]
Reordering the sums in $T_1\exm$ yields $T_1\exm=T_{1,1}\exm+T_{1,2}\exm$ with
\[\begin{array}{l} \displaystyle
T_{1,1}\exm=-\sum_{n=0}^{N\exm-2} \delta t\exm \sum_{K\in\mesh\exm} |K|\ \rho_K^{n+2} e_K^{n+2}\ \frac{\varphi_K^{n+1}-\varphi_K^n}{\delta t\exm},
\\[4ex] \displaystyle
T_{1,2}\exm=-\sum_{K\in\mesh\exm} |K|\ \rho_K^1 e_K^1\ \varphi_K^0.
\end{array}\]
By a proof similar to that of Step 2 in Lemma \ref{lmm:dt_term}, we get:
\[
\lim_{m \to +\infty} T_{1,1}\exm = -\int_0^T \int_\Omega \bar \rho\,\bar e\ \partial_t \varphi \dx \dt.
\]
The treatment of the term $T_{1,2}\exm$ is more intricate, since $\rho_K^1$ and $e_K^1$ are not directly related to the initial conditions; this is the reason why the $W^{1,1}$ regularity assumption on the initial conditions is needed.
Indeed, using the discrete internal energy balance \eqref{eq:scheme_Eint}, we get:
\[\begin{array}{l}\displaystyle
|| \rho^1 e^1-\rho^0 e^0||_{L^1(\Omega)} \leq \delta t\exm \ (||\dive_\mesh(\rho^0 e^0 \bfu^0)|| + ||p^0\, \dive(\bfu^0)||), \mbox{ with }
\\[2ex] \displaystyle \qquad
||\dive_\mesh(\rho^0 e^0 \bfu^0)||=\sum_{K\in\mesh\exm} |\sum_{\edge \in\edges(K)} F_{K,\edge}^n e^n_\edge|,
\\[2ex] \displaystyle \qquad
||p^0\,\dive(\bfu^0)||= \sum_{K\in\mesh\exm}|K|\ | p^n_K \,(\dive \bfu)^n_K|.
\end{array}\]
These last two terms are bounded if $\rho_0, e_0 \in W^{1,1}(\Omega)$ and $\bfu_0 \in W^{1,1}(\Omega)^d$ and thus, when $m$ converges to $+\infty$, $\rho^1\, e^1$ tends to $\rho_0\, e_0$ in $L^1(\Omega)$ and
\[
\lim_{m \to +\infty} T_{1,2}\exm = -\int_\Omega \rho_0\,e_0\ \varphi(\bfx,0) \dx.
\]

\medskip
The convergence of the sequence of discrete velocities in $L^1(\Omega\times(0,T))^d$ and its boundedness in $L^\infty(\Omega\times(0,T))^d$ yield that the sequence of discrete functions associated to the kinetic energy also converges in $L^1(\Omega\times(0,T))$ and is uniformly bounded, and that the limit is equal to $\frac 1 2 \,|\bar u|^2$.
In addition, thanks to the fact that the initial velocity $\bfu_0$ belongs to $L^\infty(\Omega)^d$, the sequence $(\frac 1 2\,|(\bfu\exm)^0|^2)_{m\in\xN}$ defined by \eqref{eq:inicond} converges to $\frac 1 2\,|\bfu_0|^2$. 
Lemma \ref{lmm:dt_term}, with $z=\Ekin$, thus shows that
\[
\lim_{m \to +\infty} T_2\exm = - \int_0^T \int_\Omega \frac 1 2\ \bar \rho\ |\bar\bfu|^2\ \partial_t \varphi \dx \dt
- \int_\Omega \rho_0\ |\bfu_0|^2\ \varphi(\bfx,0)\dx.
\]

\medskip
Standard arguments combining Lemmas \ref{lmm:trans} and \ref{lmm:gradw} show that
\[
\lim_{m \to +\infty} T_3\exm = -\int_0^T \int_\Omega \bar\rho\ \bar e\ \bar\bfu \cdot \gradi \varphi \dx \dt,\quad
\lim_{m \to +\infty} T_5\exm = -\int_0^T \int_\Omega \bar p\ \bar\bfu \cdot \gradi \varphi \dx \dt.
\]

\medskip
Let $(\widetilde \Gkin)_\edge^n$ be defined by:
\[
(\widetilde \Gkin)_\edge^n= - \sum_{\substack{\edged \in \edges(D_\edge)\\ \edged \subset K}} F_{\edge,\edged}^n\ (\Ekin)_\edged^n
+ \sum_{\substack{\edged \in \edges(D_\edge)\\ \edged \not \subset K}} F_{\edge,\edged}^n\ (\Ekin)_\edged^n,
\]
where, for $\edged=\edge|\edgeprim$, the quantity $(\Ekin)_\edged^n$ may be any convex combination of $(\Ekin)_\edge^n$ and $(\Ekin)_{\edgeprim}^n$, for instance $(\Ekin)_\edge^n$ with $D_\edge$ the upwind cell of $\edged$ with respect to $F_{\edge,\edged}^n$.
Let us define $\widetilde T_4\exm$ by:
\[
\widetilde T_4\exm=\sum_{n=0}^{N\exm-1} \delta t\exm \sum_{K\in\mesh\exm} \ \varphi_K^n  \sum{\edge \in \edges(K)} (\widetilde \Gkin)_\edge^n,
\]
Then, the convergence of the sequence of discrete kinetic energies already invoked for the term $T_2\exm$ implies that, thanks to Lemma \ref{lmm:div_term},
\[
\lim_{m \to +\infty} \widetilde T_4\exm = - \int_0^T \int_\Omega \frac 1 2 |\bar \bfu|^2\ \bar\bfu \cdot \gradi \varphi \dx \dt.
\]
Let us write $T_4\exm = \widetilde T_4\exm + R_4\exm$ with
\[
R_4\exm = \sum_{n=0}^{N\exm-1} \delta t\exm \sum_{K\in\mesh\exm} \ \varphi_K^n  \sum{\substack{\edge \in \edges(K)\\\edge=K|L}}  \Bigl[ (\Gkin)_\edge^n - (\widetilde \Gkin)_\edge^n \Bigr].
\]
The fact that $R_4\exm$ tends to zero when $m$ tends to $+\infty$ follows from the assumed boundedness of the unknowns and Lemma \ref{lmm:trans} and concludes the proof.
\end{proof}

\begin{remark}[On the $W^{1,1}$ assumption on the initial data]
	Note that the $W^{1,1}$ assumption on the initial data is needed because of the time shift on the pressure terms, and because we have taken into account the initial condition. 
	However, if we take the test function $\varphi \in C_c(\Omega\times(0,T))$ instead of $\varphi \in C_c(\Omega\times[0,T))$, then this assumption is not needed. 
	In particular, the  $W^{1,1}$ assumption on the initial data is not needed to recover the Rankine-Hugoniot condition.
\end{remark}

%
%
\bibliographystyle{amsplain}
\bibliography{cons} 
\end{document}